\documentclass[12pt, reqno]{amsproc}
\usepackage[utf8]{inputenc}
\usepackage[english]{babel}
\usepackage[T2A]{fontenc}
\usepackage[margin=3.1cm]{geometry}
\usepackage{amssymb, amsmath, amscd, amsthm}

\title{Topological stability of continuous functions with respect to averagings}
\author{Sergiy Maksymenko, Oksana Marunkevych}
\address{Institute of Mathematics of NAS of Ukraine, Str. Tereshchenkivs'ka, 3, 01601, Kyiv, Ukraine}
\email{maks@imath.kiev.ua, oxanamarunkevysh@rambler.ru}


\newcommand\RRR{\mathbb{R}}
\newcommand\eps{\varepsilon}

\newtheorem{definition}[subsection]{Definition}
\newtheorem{theorem}[subsection]{Theorem}
\newtheorem{lemma}[subsection]{Lemma}
\newtheorem{remark}[subsection]{Remark}

\newtheorem{counterexample}[subsection]{Couonterexample}

\newcommand\fL{f_{L}}
\newcommand\fR{f_{R}}

\begin{document} 

\thispagestyle{plain}
\maketitle

\begin{abstract}
We present sufficient conditions for topological stability of continuous functions $f:\mathbb{R}\to\mathbb{R}$ having finitely many local extrema with respect to averagings by discrete measures with finite supports.
\end{abstract}

\section{Introduction}
In applied problems of signals processing, images restoring and digitizing, noise removing, etc., a crucial role is played by linear filters.
If $x(t)$ is a signal, then the result of the application of a linear filter with impulse response $h(t)$ on time interval $[0,T]$ is a convolution $x*h$ of these functions, i.e. a signal defined by the following formula:
\[
x*h(t) = \int_{0}^{T} x(t-\tau) h(\tau) d\tau.
\]
In the case when the support of $h$ is sufficiently small and $\int_{0}^{T} h(\tau)d\tau = 1$, the function $h$ can be regarded as a density of some measure, while the convolution $x*h$ can be viewed as an \textit{averaging} of $x$ with respect to this measure.
Such averagings are widely used in applications, see e.g.~\cite{AkhmanovDyakovChirkin:1981}, \cite{Kenneth:IEEE:1995}, \cite{Milanfar:IEEE:2013}.

Notice that a priori the ``form'' of a averaged signal $x*h$ can be essentially different of the form of the initial signal $x$.
For instance, if $x$ has a unique maximum point, then $x*h$ may have many maximums.
``Preserving form'' is a principal requirement to filters in the problems of noise removing, computing entropies of time series, and others, see e.g.~\cite{BandtPompe:PRL:2002}, \cite{AntonioukKellerMaksymenko:CDCSA:2014} and references in these papers.

From mathematical point of view <<similarity of forms>> of signals means that they are \textit{topologically equivalent} as functions of time, see Definitions~\ref{def:local_top_equiv} and~\ref{def:top_equiv} below. 

In the present paper we give wide sufficient conditions for topological stability of averagings of piece wise differentiable functions $f:\RRR\to\RRR$ having finitely many local extrema with respect to discrete measures with finite supports, see Theorems~\ref{th:main_theorem_global} and~\ref{th:main_theorem_local}.

Those conditions guarantee that after applying to a signal $x(t)$ a linear filter with impulse response $h(t)$ being a sum of finitely many $\delta$-functions the form of the resulting signal $x*h$ will not change.

\medskip

\section{Averagings of a function}
Let $\mu$ be any probability measure on the closed segment $[-1,1]$.
This means that $\mu$ a non-negative $\sigma$-additive measure defined on the Borel algebra of subsets of $[-1,1]$ and such that $\mu[-1,1]=1$.
Then for each measurable function $f:\RRR\to\RRR$ and a positive number $\alpha>0$ one can define the function $f_{\alpha}:\RRR\to\RRR$ by the following formula:
\begin{equation}\label{equ:averaging}
f_{\alpha}(x) = \int_{-1}^{1} f(x-t\alpha)  d\mu.
\end{equation}
We will call $f_{\alpha}$ an \textit{$\alpha$-averaging} of $f$ with respect to the measure $\mu$.

Notice that if $f$ is defined only on some interval $(a,b)$ and $2\alpha < b-a$, then the formula~\eqref{equ:averaging} determines a function $f_{\alpha}$ on the interval $(a+\alpha, b-\alpha)$.
Moreover,
\begin{equation}\label{equ:estimation_for_falpha}
\inf_{y\in[x-\alpha,x+\alpha]} f(y) \ \ \leq \ \ f_{\alpha}(x) \ \ \leq \ \ \sup_{y\in[x-\alpha,x+\alpha]} f(y).
\end{equation}

Consider few simple cases.

1) Suppose $\mu$ is a discrete measure with finite support.
This means that there exists a finite increasing sequence of points $t_k < t_{k-1} < \ldots < t_2 < t_1 \in [-1,1]$ such that for arbitrary Borel set $A\subset[-1,1]$ its measure is given by
\[
\mu(A) = \sum_{t_i \in A} \mu(t_i).
\]
Then the formula for $\alpha$-averaging of $f:\RRR\to\RRR$ can be represented as follows:
\begin{equation}\label{equ:falpha_discr_mu}
f_{\alpha}(x) = \sum_{i=1}^{k} f(x-t_i\alpha)\,\mu(t_i).
\end{equation}
In particular, if $k=2$, $t_1=-1$, $t_2=+1$ and $\mu(-1)=\mu(1)=\tfrac{1}{2}$, then 
\begin{equation}\label{equ:ariphmetic_average}
f_{\alpha}(x) = \frac{f(x+\alpha)+f(x-\alpha)}{2}.
\end{equation}

2) Suppose that $\mu$ is absolutely continuous, so there exists a measurable function $p:[-1,1]\to[0,\infty)$ such that $\mu(A) = \int_{A} p(t) dt$.
Then
\[
f_{\alpha}(x) = \int_{-1}^{1} f(x-t\alpha) p(t) dt.
\]

\begin{lemma}\label{lm:prop_falpha}
The correspondence $f \mapsto f_{\alpha}$ is a linear operator on the space of all continuous functions $C(\RRR,\RRR)$.
Suppose that $f \in C(\RRR,\RRR)$ has one of the following properties: $f$ is positive, non-negative, negative, non-positive, (strictly) increase, (strictly) decrease, (strictly) convex, (strictly) concave.
Then the same property has the $\alpha$-averaging $f_{\alpha}$ for each $\alpha>0$.
\end{lemma}
\begin{proof}
We consider only the cases of (strictly) increasing and convex functions.
All other statements are either obvious or can be proved in a similar way and we leave them for the reader.

1) Suppose $f$ increases, so for all $x<y\in\RRR$, $t\in[-1,1]$, and $\alpha>0$ we have that $f(x-t\alpha) \leq f(y-t\alpha)$.
Therefore
\[
f_{\alpha}(x) = \int_{-1}^{1} f(x-t\alpha)  d\mu \leq \int_{-1}^{1} f(y-t\alpha)  d\mu = f_{\alpha}(y),
\]
that is $f_{\alpha}$ also increases.

2) If $f$ strictly increases, that is $q(t) = f(y-t\alpha) - f(x-t\alpha) > 0$ for all $t\in[-1,1]$, then
\[
f_{\alpha}(y) - f_{\alpha}(x) = \int_{-1}^{1} q(t) d\mu >0,
\]
since the measure $\mu$ is non-negative.
Hence $f_{\alpha}$ is also strictly increasing.

3) Suppose that $f$ is convex, that is for all $x,y\in\RRR$ and $s\in[0,1]$ we have that
\[
f(sx+(1-s)y) \leq sf(x) +(1-s)f(y).
\]
Then
\begin{align*}
f_{\alpha}(sx+(1-s)y) &=
 \int_{-1}^{1} f(sx+(1-s)y-t\alpha)  d\mu  \\
 & \leq 
 s \int_{-1}^{1} f(x-t\alpha)  d\mu +
(1-s) \int_{-1}^{1} f(y-t\alpha)  d\mu \\ 
&=
sf_{\alpha}(x) +(1-s)f_{\alpha}(y).
\end{align*}
Lemma is completed.
\end{proof}

\begin{lemma}\label{lm:inf_limit_falpha}
Let $f:(a,+\infty)\to\RRR$ be a continuous and strictly monotone function.
Then $\lim\limits_{x\to+\infty}f_{\alpha}(x)= \lim\limits_{x\to+\infty}f(x)$ for all $\alpha>0$. 
\end{lemma}
\begin{proof}
For definiteness assume that $f$ strictly increases.
Then it follows from formula~\eqref{equ:estimation_for_falpha} that for $x>\alpha$ the following inequalities hold:
\[ 
f(x-\alpha) \ = \ \inf_{y\in[x-\alpha,x+\alpha]} f(y) \ \ \leq \ \ f_{\alpha}(x) \ \ \leq \ \ \sup_{y\in[x-\alpha,x+\alpha]} f(y) \ = \ f(x+\alpha),
\]
whence $\lim\limits_{x\to+\infty}f_{\alpha}(x)= \lim\limits_{x\to+\infty}f(x)$.
\end{proof}

\section{Topological equivalence of functions}
At first we will recall the notion of a \textit{germ} of a function.
Let $a\in\RRR$, $U$ be a neighborhood of $a$, and $f,g:U\to\RRR$ be two continuous functions.
Then $f$ and $g$ determine the same \textit{germ} at $a$ whenever $f=g$ on some neighborhood $V\subset U$ of $a$.
The relation ``define the same germ at $a$'' is obviously an equivalence, and the corresponding equivalence classes are called \textit{germs} at $a$.
We will denote the class of $f:U \to\RRR$ at $a$ by $f:(\RRR,a)\to\RRR$ or by $f:(\RRR,a)\to(\RRR, f(a))$ if we wan to specify the value of $f$ at $a$.

Recall also that a homeomorphism $\phi:(a,b) \to (c,d)$ is the same as a continuous surjective strictly monotone function.
Moreover, if $\phi$ increases (decreases) then $\phi$ is said to \textit{preserve (reverse) orientation}.

\begin{definition}\label{def:local_top_equiv}
Let $a,b\in \RRR$ and $f:(\RRR,a)\to\RRR$ and $g:(\RRR,b)\to\RRR$ be two germs of continuous functions at $a$ and $b$ respectively.
Then $f$ and $g$ are called {\bfseries topologically equivalent} if there exist two germs of orientation preserving homeomorphisms $h:(\RRR,a) \to (\RRR,b)$ and $\phi:(\RRR,f(a)) \to (\RRR,g(b))$ such that $\phi \circ f = g \circ h$.
\end{definition}

\begin{remark}\rm
In the definition of topological equivalence it is not necessary to assume that $\phi$ and $h$ preserve orientation.
However in the present paper we will always do this.
\end{remark}

The following simple lemma is left for the reader.
\begin{lemma}\label{lm:local_top_equiv}
Let $f:(\RRR,a)\to\RRR$ and $g:(\RRR,b)\to\RRR$ be two germs of continuous functions.
Suppose also that one of the following conditions holds true:
\begin{enumerate}
\item[\rm(1)]
$f$ and $g$ are strictly monotone on some neighborhoods of $a$ and $b$ respectively;
\item[\rm(2)]
the points $a$ and $b$ are isolated local maximums (resp. local minimums) of $f$ and $g$ respectively.
\end{enumerate}
Then $f$ and $g$ are topologically equivalent.
\end{lemma}

\begin{definition}\label{def:top_equiv}
Two continuous functions $f:(a,b) \to \RRR$ and $g:(c,d) \to \RRR$ are called {\bfseries topologically equivalent} if there exist orientation preserving homeomorphisms 
\begin{align*}
h: & \ (a,b)\to(c,d), &
\phi: & \ \RRR\to\RRR,
\end{align*}
such that $\phi\circ f = g \circ h$, that is they made commutative the following diagram:
\[
\begin{CD}
(a,b) @>{f}>{}>\RRR \\
@V{h}VV @VV{\phi}V \\
(c,d) @>{g}>> \RRR
\end{CD}
\]
\end{definition}

We will now recall some results about classification of continuous functions on the real line up to a topological equivalence.
\begin{definition}\label{def:snake}{\rm\cite{Arnold:UMN:1992}}
A {\bfseries generalized snake} of length $k$, or simply a {\bfseries $k$-snake} is an arbitrary sequence of $k$ numbers $\{A_1,\ldots,A_k\}$.
Two $k$-snakes $\{A_1,\ldots,A_k\}$ and $\{B_1,\ldots,B_k\}$ are {\bfseries equivalent} whenever for any $i,j=1,\ldots,k$ the following condition holds true:
\begin{enumerate}
\item[\rm(*)]
$A_i < A_j$ if an only if $B_i < B_j$;
\end{enumerate}
\end{definition}
Evidently, this condition also implies that $A_i = A_j$ if and only if $B_i = B_j$.

Let $f:\RRR\to\RRR$ be a continuous function having only finitely many local extremes $x_1,\ldots,x_n$ and being strictly monotone on the complementary intervals to these points.
In particular, it follows that there exist finite or infinite limits $A_0=\lim\limits_{x\to-\infty} f(x)$ and $A_{n+1}=\lim\limits_{x\to+\infty} f(x)$.
Denote $A_i = f(x_i)$, $i=1,\ldots,n$.
Then the sequence of numbers $\xi(f) = \{A_0,\ldots,A_{n+1}\}$ will be called a \textit{snake associated with $f$}.

The following statement is well-known and can be easily proved.
\begin{lemma}\label{lm:classification_by_snakes}{\rm e.g.~\cite{Thom:Top:1965}, \cite{Arnold:UMN:1992}}
Let $f,g:\RRR\to\RRR$ be continuous functions both having exactly $k$ local extremes for some $k\geq0$ and being strictly monotone on the complementary intervals to these points.
Then $f$ and $g$ are topologically equivalent if and only if the corresponding snakes $\xi(f)$ and $\xi(g)$ are equivalent.
\qed
\end{lemma}

\begin{definition}\label{def:top_stab_rel_averagings}
Let $f:\RRR\to\RRR$ be a continuous function and $\mu$ be a probability measure on $[-1,1]$.
We will say that $f$ is {\bfseries topologically stable} with respect to the averagings by measure $\mu$ whenever there exists $\varepsilon>0$ such that for all $\alpha\in(0,\varepsilon)$ the functions $f$ and $f_{\alpha}$ are topologically equivalent.
\end{definition}

Similarly, one can give a definition of a local topological stability of averagings by measure $\mu$.
Let $f:(\RRR,a)\to\RRR$ be a germ of a continuous function at a point $a\in\RRR$.
This means that $f$ is a continuous function defined on the interval $(a-\varepsilon, a+\varepsilon)$ for some $\varepsilon>0$.
Then it follows from~\eqref{equ:averaging} that for $\alpha<\varepsilon/2$ the averaging $f_{\alpha}$ is correctly defined on the interval $(a-\varepsilon/2, a+\varepsilon/2)$.
Moreover, the germ of $f_{\alpha}$ at $a$, evidently, depends only on the germ of $f$ at that point.

\begin{remark}\rm
Notice the germs $f$ and $f_{\alpha}$ at $a$ are in general not topologically equivalent.
For example, if $a$ is an isolated local minimum of $f$, then $f_{\alpha}$ may also have an isolated local minimum $b$ very closed to $a$ but distinct from $a$.
Then the germs of $f$ and $f_{\alpha}$ at $a$ are not topological equivalent, though by Lemma~\ref{lm:local_top_equiv} the restriction of $f$  on some neighborhood $(c_1,c_2)$ of $a$ will be topologically equivalent to the restriction of $f_{\alpha}$ to some neighborhood $(d_1,d_2)$ of $b$.
This observation leads to the following definition.
\end{remark}

\begin{definition}\label{def:loc_top_stab}
A germ $f:(\RRR,a)\to\RRR$ is said to be {\bfseries topologically stable} with respect to averagings by measure $\mu$ if there exists $\varepsilon>0$ such that for each $\alpha\in(0,\varepsilon)$ the following condition holds true:
\begin{itemize}
\item[]
there exist $c_1,c_2, d_1,d_2\in(a-\eps, a+\eps)$ depending on $\alpha$ and such that $c_1 < a < c_2$, $d_1 < d_2$, and the restrictions 
\begin{align*}
f|_{(c_1,c_2)}: & \ (c_1,c_2) \ \to \ \RRR, &
f_{\alpha}|_{(d_1,d_2)}: & \ (d_1,d_2) \ \to \ \RRR
\end{align*}
are topologically equivalent.
\end{itemize}
\end{definition}

In the present paper we give sufficient conditions for topological stability of averagings of piece-wise differentiable functions with respect to averaging by discrete probability measures with finite supports.

The following theorem shows that for functions of ``general position'' with finitely many local extremes a local stability with respect to averagings by measure $\mu$ implies global stability.

\begin{theorem}\label{th:main_theorem_global}
Let $\mu$ be a probability measure on $[-1,1]$ and $f:\RRR \to \RRR$ be a continuous function having only finitely many local extremes $x_1,\ldots,x_n$ and being strictly monotone on the complement to these points.
As above denote 
\begin{align*}
A_{0}&=\lim\limits_{x\to-\infty} f(x), &
A_i&=f(x_i), \ i=1,\ldots,n, &
A_{n+1}&=\lim\limits_{x\to+\infty} f(x).
\end{align*}
Suppose that the following conditions hold:
\begin{enumerate}
\item[\rm(1)]
the numbers $A_1,\ldots,A_n$ are mutually distinct and differ from $A_0$ and $A_{n+1}$ as well;
\item[\rm(2)]
for each $i=1,\ldots,n$ the germ $f:(\RRR,x_i) \to \RRR$ of $f$ at $x_i$ is topologically stable with respect to averagings by $\mu$.
\end{enumerate}
Then $f$ is topologically stable with respect to averagings by $\mu$.
\end {theorem}
\begin{proof}
It suffices to find $\varepsilon>0$ such that the snakes $\xi(f)$ and $\xi(f_{\alpha})$ are equivalent for all $\alpha\in(0,\varepsilon)$.
Then it will follows from Lemma~\ref{lm:classification_by_snakes} that $f$ and $f_{\alpha}$ are topologically equivalent, whence $f$ will be topologically stable with respect to averagings by $\mu$.

Since $f$ has only finitely many local extremes, it follows from (2) and Lemma~\ref{lm:prop_falpha} that there exists $\varepsilon>0$ such that for all $\alpha\in(0,\varepsilon)$ the averaging $f_{\alpha}$ has also exactly $n$ local extremes.
Let $\xi(f_{\alpha}) = \{B_0,B_1,\ldots,B_{n+1}\}$ be the corresponding snake for $f_{\alpha}$.

Then by Lemma~\ref{lm:inf_limit_falpha} $A_0 = B_0$ and $A_{n+1} = B_{n+1}$.
Moreover, it follows from inequalities~\eqref{equ:estimation_for_falpha} that one can reduce $\varepsilon$ so that for each pair $i\not=j$ the condition $A_i<A_j$ will imply that $B_i < B_j$ as well and that $B_i$ also differs from $B_0$ and $B_{n+1}$.
This implies that the snakes $\xi(f)$ and $\xi(f_{\alpha})$ are equivalent.
\end{proof}

\section{Topological stability of germs with respect to averagings}
Let $\eps>0$ and $f:(-\eps, \eps)\to\RRR$ be a continuous function such that $0$ is an isolated local minimum for $f$ and that $f$ monotone decreases on $(-\eps,0]$ and monotone increases on $[0,\eps)$.
It will be convenient to denote
\begin{align*}
\fL &= f|_{(-\eps,0]}: (-\eps,0] \longrightarrow \RRR, &
\fR &= f|_{[0, \eps)}: [0, \eps) \longrightarrow \RRR.
\end{align*}

\begin{lemma}\label{lm:convex_functions_loc_stab}
Let $\mu$ be a probability measure on $[-1,1]$. 
Then each of the following conditions implies that the germ of $f$ at $0$ is locally stable with respect to averagings by $\mu$:
\begin{enumerate}
\item[\rm(1)]
$f$ is strictly convex;
\item[\rm(2)]
$f$ is $C^{1}$-differentiable $(-\eps,0) \cup (0,+\eps)$ and $f'$ is strictly increasing on $(-\eps,0)\cup(0,+\eps)$;
\item[\rm(3)]
$f$ is $C^{2}$-differentiable on some neighborhood of $0$ and $f''(0)>0$.
\end{enumerate}
\end{lemma}
\begin{proof}
Evidently, (3)~$\Rightarrow$~(2)~$\Rightarrow$~(1), so it suffices to prove (1).

(1) Suppose $f$ is strictly convex and let $\alpha<2\eps$.
Then by Lemma~\ref{lm:prop_falpha} the averaging $f_{\alpha}$ is also strictly convex, whence it has a unique minimum point as well as $f$.
Then by Lemma~\ref{lm:local_top_equiv} $f$ and $f_{\alpha}$ are topological equivalent and so the germ of $f$ at $0$ is topological stable with respect to averagings by measure $\mu$.
\end{proof}

\begin{remark}\rm
One can assume that in (3) of Lemma~\ref{lm:convex_functions_loc_stab} the homeomorphisms $h$ and $\phi$ satisfying $\phi \circ f = f_{\alpha} \circ h$ are \textit{diffeomorphisms}.
Indeed, the assumption that $f$ belongs to class $C^2$ near $0$ and $f''(0)>0$ means $0$ is a non-degenerate critical point.
Moreover, for all small $\alpha>0$ the function $f_{\alpha}$ will also belong to class $C^2$ and also will have a unique minimum point, say  $x_{\alpha}$, with $f''_{\alpha}(x_{\alpha})>0$.
Therefore $x_{\alpha}$ is a non-degenerate critical point for $f_{\alpha}$ as well.
Then by Morse Lemma the germs $f:(\RRR,0)\to\RRR$ and $f_{\alpha}:(\RRR,x_{\alpha})\to\RRR$ are smoothly equivalent, that is $h$ and $\phi$ can be chosen to be diffeomorphisms, see~\cite[Theorem~II.6.9, Proposition~III.2.2]{GolubitskiGuillemin:StableMats:ENG:1973}.
\end{remark}

\section{Main result}
Let $\mu$ be a probability measure on $[-1,1]$ with finite support $t_k<t_{k-1}<\cdots<t_1$.
Put $p_i = \mu(t_i)$, $i=1,\ldots,k$.
Then $p_i>0$ and $p_k+\cdots+p_1 = 1$.

In what follows we will assume that the function $f:(-\eps, \eps)\to\RRR$ belongs to the class $C^1$ on $(-\eps,0) \cup (0,+\eps)$ and there exist finite or infinite limits
\begin{align}\label{equ:limits_LR}
L &= \lim_{x\to0-0} \fL'(x), & R &= \lim_{x\to0+0} \fR'(x).
\end{align}
Evidently, $L\leq0$ and $R\geq0$.

If $L$ and $R$ are finite, then for each $j=1,\ldots,k-1$ we introduce the following numbers:
\begin{align}\label{equ:numbers_Xj}
X_j =& L (p_1+\cdots+ p_{j}) + R (p_{j+1} + \cdots + p_k).
\end{align}
It is easy to see that they satisfy the following inequality:
\[
L \ \leq \ X_{k-1} \ < \ X_{k-2} \ < \ \cdots \ < \ X_{1} \ \leq  \ R.
\]

\begin{theorem}\label{th:main_theorem_local}
Suppose that one of the following conditions holds true:
\begin{itemize}
\item[\rm(a)] 
both limits $L$ and $R$ are finite and $X_j \not=0$ for $j=1,\ldots,k-1$.
\item[\rm(b)] 
one of the limits either $L$ or $R$ is infinite and the other one is finite.
\end{itemize}
Then the germ of $f$ at $0$ is topologically stable with respect to averagings by measure $\mu$.
\end{theorem}

In order to make the proof of Theorem~\ref{th:main_theorem_local} more clear we will first formulate and prove a special case.
\begin{lemma}\label{lm:main_result_ariphmetic_average}
Let $\mu$ be a discrete measure on $[-1,1]$ such that $\mu(-1) = \mu(1) = \tfrac{1}{2}$ and so
\[ f_{\alpha}(x) = \frac{f(x+\alpha)+f(x-\alpha)}{2} = 
\frac{1}{2} \cdot 
\begin{cases}
\fL(x+\alpha)+\fL(x-\alpha), &  x \in (-\eps+\alpha, -\alpha), \\
\fL(x+\alpha)+\fR(x-\alpha), & x \in [-\alpha, \alpha], \\
\fR(x+\alpha)+\fR(x-\alpha), & x \in (\alpha, \eps-\alpha).
\end{cases}
\]
see Eq.~\eqref{equ:ariphmetic_average}.
Suppose that both limits $L$ and $R$ are finite and $L+R \not=0$.
Then germ of $f$ at $0$ is topologically stable with respect to the averagings by $\mu$.
\end{lemma}

Notice that under the assumptions on $\mu$ we have in Lemma~\ref{lm:main_result_ariphmetic_average} that $k=2$, $p_2 = p_1 = \tfrac{1}{2}$.
Hence 
\begin{align}\label{equ:flplusfr_non_zero}
X_1 &=\tfrac{1}{2}(L + R) \not=0,
\end{align}
Thus Lemma~\ref{lm:main_result_ariphmetic_average} is a particular case of Theorem~\ref{th:main_theorem_local}.

\begin{proof}[Proof of Lemma~\ref{lm:main_result_ariphmetic_average}]
It suffices to find $\delta>0$ such that for all $\alpha\in(0,\delta/2)$ the function $f_{\alpha}$ will have a unique minimum as well as $f$.
Then by Lemma~\ref{lm:local_top_equiv} $f$ and $f_{\alpha}$ will be topologically equivalent.

Since the limits~\eqref{equ:limits_LR} are finite, the derivatives $\fL'$ and $\fR'$ continuously extend to the closed segments $(-\eps,0]$ and $[0,\eps)$ respectively, so that $\fL'(0)=L$ and $\fR'(0)=R$.
Then we get from~\eqref{equ:flplusfr_non_zero} and continuity of $\fL'$ and $\fR'$ it follows that there exists $\delta\in(0,\eps)$ such that for any $x,y\in(0,\delta)$ the following inequality holds:
\[\fL'(-x) + \fR'(y) \not=0.\]
By assumption $\fL$ is monotone decreasing on $(-\eps, 0)$ while $\fR$ is monotone increasing on $(0, \eps)$.
Hence we get from Lemma~\ref{lm:prop_falpha} that $f_{\alpha}$ is monotone increasing on $(-\eps+\alpha, -\alpha)$ and monotone decreasing on $(\alpha, \eps-\alpha)$.
We will show that $f_{\alpha}$ is strictly monotone on $(-\alpha, \alpha)$.
Therefore $f_{\alpha}$ will have a minimum point at one of the ends of the segment $[-\alpha, \alpha]$ in accordance with the sign of $L+R$.

For definiteness assume that $L+R>0$, whence $\fL'(-x) + \fR'(y) > 0$ for all $x,y\in(0,\delta)$.
Therefore for $\alpha\in(0, \delta/2)$ and $x \in (-\alpha, \alpha) \subset (-\delta/2, \delta/2)$ we have that
\[
-\delta \ < \ x-\alpha \ < \ 0 \  < \ x+\alpha \ < \ \delta,
\]
whence
\[ f'_{\alpha}(x)  = \fL'(x+\alpha)+\fR'(x-\alpha) > 0. \]
Thus $f_{\alpha}$ is strictly monotone on $[-\alpha, \alpha]$.
On the other hand, $f_{\alpha}$ is strictly increasing on $(-\eps+\alpha,-\alpha]$ and strictly increasing on $[\alpha,\eps-\alpha]$.
Therefore $f_{\alpha}$ has a unique minimum point $x=-\alpha$.
\end{proof}

Let us also show that the assumption~\eqref{equ:flplusfr_non_zero} is essential. 

\begin{counterexample}\rm
Let $f(x) = |x|$.
Then $\fL(x) =-x$, $\fR(x) =x$, $L = \fL'(0) = -1$ and $R = \fR'(0) = +1$, whence
\[
X_1 = L + R  = -1+1 = 0,
\]
so the condition~\eqref{equ:flplusfr_non_zero} fails. 
In this case for every $\alpha>0$
\[ f_{\alpha}(x) = \frac{1}{2}\bigl(|x+\alpha|+|x-\alpha|\bigr) = 
\begin{cases}
-x, &  x \in (-\infty, -\alpha), \\
\alpha, & x \in [-\alpha, \alpha], \\
x, & x \in (\alpha, +\infty).
\end{cases}
\]
Thus $f_{\alpha}$ is constant on the segment $[-\alpha, \alpha]$, whence it is not topologically equivalent to $f$.
\end{counterexample}

\section{Proof of Theorem~\ref{th:main_theorem_local}}
If $k=1$, then $f_{\alpha}(x) = f(x-t_1\alpha)$, whence $f$ and $f_{\alpha}$ are topologically equivalent.

So assume that $k\geq2$.
It suffices to show that there exists $\delta>0$, such that for $\alpha\in(0,\delta/2)$ the function $f_{\alpha}$ has a unique minimum point as well as $f$.
Then by Lemma~\ref{lm:local_top_equiv} $f$ and $f_{\alpha}$ will be topologically equivalent.

Notice that $f_{\alpha}$ is given by the following formulas:
\begin{equation}\label{equ:falpha_general_case}
f_{\alpha}(x) = 
\left\{\begin{array}{ll}
\sum\limits_{i=1}^{k} \fL(x-t_i\alpha) p_i, & \ x \in (-\eps+\alpha, t_k\alpha), \\ [4mm]
\sum\limits_{i=1}^{j} \fL(x-t_i\alpha) p_i + \sum\limits_{i=j+1}^{k} \fR(x-t_i\alpha) p_i, & \ x \in [t_{j+1}\alpha, t_{j}\alpha), \\
 & \ k-1 \geq j \geq 1,\\ [4mm]
\sum\limits_{i=1}^{k} \fR(x-t_i\alpha) p_i, & \ x \in [t_1\alpha, \eps-\alpha).
\end{array}\right.
\end{equation}
Indeed, it follows from the condition $x \in (-\eps+\alpha, t_k\alpha)$ that $x-t_k\alpha<0$, and therefore $x-t_i\alpha<0$ for all $i=1,\ldots,k$.
Hence the value $f_{\alpha}(x)$ is given by the first line of formula~\eqref{equ:falpha_general_case}.

Further, the assumption $x \in [t_{j+1}\alpha, t_{j}\alpha)$ is equivalent to $x-t_{j}\alpha < 0 \leq x-t_{j+1}\alpha$, whence
\[
x-t_1\alpha \ < \ \cdots \ < \ x-t_{j}\alpha \ < \  0 \ \leq \ x-t_{j+1}\alpha \ < \ \cdots \ < \ x-t_k\alpha.
\]
Therefore $f_{\alpha}(x)$ is given by the second line of formula~\eqref{equ:falpha_general_case}.

Similarly, it follows from the assumption $x \in [t_1\alpha, \eps-\alpha)$ that $x-t_i\alpha\geq0$ for all $i=1,\ldots,k$, and therefore the value $f_{\alpha}(x)$ is given by the third line of~\eqref{equ:falpha_general_case}.

\medskip
By assumption $\fL$ is monotone decreasing on $(-\eps, 0)$, while $\fR$ is monotone decreasing on $(0, \eps)$.
Then by Lemma~\ref{lm:prop_falpha} $f_{\alpha}$ is monotone decreasing on $(-\eps+\alpha, t_k\alpha)$ and monotone increasing on $(t_1\alpha, \eps-\alpha)$.
We will show that for some $m\in\{1,\ldots,k\}$ the function $f_{\alpha}$ is strictly decreasing on $(t_k\alpha, t_m\alpha)$ and strictly increasing on $(t_m\alpha, t_1\alpha)$.
This will imply that $f_{\alpha}$ has a unique minimum point $t_m\alpha$.

For each $j=k-1,\ldots,2,1$ define a function $g_j:(0,\eps)^k \to \RRR$ by the following formula:
\begin{align*}
g_j(x_1,\ldots,x_k) &= \sum_{i=1}^{j} \fL'(-x_i)\, p_i \ + \ \sum_{i=j+1}^{k} \fR'(x_i)\, p_i.
\end{align*}

\begin{lemma}\label{lm:suff_cond_for_stab}
Suppose that there exist $\delta\in(0,\eps)$ and $m\in\{1,\ldots,k\}$ such that for all $(x_1,\ldots,x_{k})\in(0,\delta)^k$ the following inequalities hold:
\begin{equation}\label{equ:signs_of_derivative}
\begin{aligned}
g_j(x_1,\ldots,x_k)&<0, \qquad j \geq m, \\
g_j(x_1,\ldots,x_k)&>0, \qquad j < m.
\end{aligned}
\end{equation}
Then for each $\alpha\in(0,\delta/2)$ the function $f_{\alpha}$ has a unique minimum point $x = t_{m}\alpha$.
\end{lemma}
\begin{proof}
Let $\alpha\in(0,\delta/2)$, $j\in\{k-1,\ldots,2,1\}$ and $x \in (t_{j+1}\alpha, t_{j}\alpha)$.
Then by formula~\eqref{equ:falpha_general_case}, 
\begin{align*}
f'_{\alpha}(x) &= \sum\limits_{i=1}^{j} \fL'(x+t_i\alpha) p_i + \sum\limits_{i=j+1}^{k} \fR'(x+t_i\alpha) p_i \\
& = g_j\bigl(\,  -x-t_1\alpha, \, \ldots, \, -x-t_j\alpha, \ x+t_{j+1}\alpha, \, \ldots, \, x+t_{k}\alpha\, \bigr).
\end{align*}
Also notice that $|x| < \alpha$, whence
\[
  |x-t_i\alpha| \leq |x| + \alpha < 2\alpha < \delta, \qquad i=1,\ldots,k.
\]
Then it follows from~\eqref{equ:signs_of_derivative} that $f'_{\alpha}(x)<0$ for $x< t_m\alpha$ and $f'_{\alpha}(x)>0$ for $x> t_m\alpha$.

Thus the derivative $f'_{\alpha}$ is defined on $(t_k\alpha, t_1\alpha)$ except possibly finitely many points of the form $t_i\alpha$, $i=1,\ldots,k$, takes negative values on $(t_k\alpha, t_m\alpha)$ and positive values of $(t_m\alpha, t_1\alpha)$.
Hence $f_{\alpha}$ has a unique minimum point $x=t_m\alpha$.
\end{proof}

It remains to check that each of the conditions (a) and (b) of Theorem~\ref{th:main_theorem_local} implies~\eqref{equ:signs_of_derivative}.

\medskip

(a) Suppose that the limits $L$ and $R$ are finite and $X_j \not=0$ for all $j=1,\ldots,k-1$.
It follows from finiteness of $L$ and $R$ that $\fL'$ and $\fR'$ extend to $[0,\eps)$ and $(-\eps,0]$ respectively by $\fL'(0)=L$ and $\fR'(0)=R$.
Therefore, see~\eqref{equ:numbers_Xj},
\begin{equation}\label{equ:numbers_Xj_func_gj}
\begin{aligned}
X_j  &= L (p_1+\cdots+ p_{j}) + R (p_{j+1} + \cdots + p_k) \\
     &= \sum_{i=1}^{j} \fL'(0)\, p_i \ + \ \sum_{i=j+1}^{k} \fR'(0)\, p_i = g_j(0,\ldots,0) \not=0.
\end{aligned}
\end{equation}

Hence by continuity of $\fL'$ and $\fR'$ there exists $\delta>0$ such that for all $(x_1,\ldots,x_{k})\in(0,\delta)^k$ and $j=1,\ldots,k-1$ the value $g_j(x_1,\ldots,x_k)$ has the same sign as $X_j$.
Due to Lemma~\ref{lm:suff_cond_for_stab} this sign also coincides with the sign of the derivative $f'_{\alpha}$ on the interval $(t_{j+1}\alpha, t_{j}\alpha)$.

Recall that $L < X_{k-1} < \cdots < X_{1} < R$ and $L \leq 0 \leq R$.
Hence there exists $m\in\{k,k-1,\ldots,1\}$ such $X_j<0$ for $k-1\geq j\geq m$ and $X_j>0$ for $1\leq j < m$.
Therefore the assumptions~\eqref{equ:signs_of_derivative} of Lemma~\ref{lm:suff_cond_for_stab} hold, whence $f$ has a unique minimum at $t_m\alpha$.
Let us explain this in more details:
\begin{enumerate}
\item[(i)]
if $0 < X_{k-1} < \cdots < X_{1}$, then $m=k$ and $f$ decreases on $(-\eps,t_k\alpha]$ and increases on $[t_k\alpha, \eps)$, and so $f$ has a minimum point $t_k\alpha$;
\item[(ii)]
if $X_{k-1} < \cdots < X_{m} < 0 < X_{m+1} < \cdots < X_{1}$, then $f$ decreases on $(-\eps,t_m\alpha]$ and increases on $[t_m\alpha, \eps)$, and so $f$ has a minimum point $t_m\alpha$;
\item[(iii)]
if $X_{k-1} < \cdots < X_{1} < 0$, then $m=1$, $f$ decreases on $(-\eps,t_1\alpha]$ and increases on $[t_1\alpha, \eps)$, whence $f$ has a minimum point $t_1\alpha$.
\end{enumerate}

\medskip

(b) Suppose that $|L|<\infty$ and $R = +\infty$.
Then one can find $\delta>0$ such that $g_j(x_1,\ldots,x_k)>0$ for all $(x_1,\ldots,x_{k})\in(0,\delta)^k$ and $j=1,\ldots,k-1$.
This means that the assumptions of Lemma~\ref{lm:suff_cond_for_stab} hold for $m=k$, whence $f_{\alpha}$ will have a unique minimum point $x=t_k\alpha$.

Similarly if $L=-\infty$ and $|R| < \infty$, then the function $f_{\alpha}$ will have a unique minimum point $x=t_1\alpha$.
Theorem~\ref{th:main_theorem_local} is completed.
\qed


\def\cprime{$'$}
\providecommand{\bysame}{\leavevmode\hbox to3em{\hrulefill}\thinspace}
\providecommand{\MR}{\relax\ifhmode\unskip\space\fi MR }
\providecommand{\MRhref}[2]{%
  \href{http://www.ams.org/mathscinet-getitem?mr=#1}{#2}
}
\providecommand{\href}[2]{#2}

\end{document}